\documentclass[11pt]{article}
\usepackage[a4paper,headinclude=false,footinclude=false,DIV=9]{typearea}
\usepackage{amsmath,amsthm}
\usepackage[pdfpagelabels,plainpages=false,colorlinks, linkcolor = blue, citecolor = blue, filecolor = blue, urlcolor = blue]{hyperref}

\usepackage[utf8]{inputenc}  
\usepackage[T1]{fontenc}
\usepackage{mathptmx} 
\usepackage[scaled=0.90]{helvet}

\newtheoremstyle{mystyle}{}{}{\rmfamily}%
{}{\normalfont\bfseries}{.}{ }{} 

\newtheorem{theorem}{Theorem}[section]
\newtheorem{corollary}[theorem]{Corollary}
\newtheorem{lemma}[theorem]{Lemma}
\newtheorem{proposition}[theorem]{Proposition}

\theoremstyle{mystyle}
\newtheorem{remark}[theorem]{Remark}

\usepackage{amssymb}
\usepackage{amsfonts}

\newcommand{\cB}{{\mathcal B}}

\newcommand{\cF}{{\mathcal F}}

\newcommand{\cI}{{\mathcal I}}

\newcommand{\cN}{{\mathcal N}}

\newcommand{\cW}{{\mathcal W}}

\newcommand{\cZ}{{\mathcal Z}}

\newcommand{\te}{{\theta}}

\newcommand{\Om}{{\Omega}}
\newcommand{\om}{{\omega}}

\newcommand{\al}{{\alpha}}

\newcommand{\ka}{{\kappa}}
\newcommand{\la}{{\lambda}}
\newcommand{\La}{{\Lambda}}

\newcommand{\bbN}{{\mathbb N}}

\newcommand{\bbR}{{\mathbb R}}

\newcommand{\with}{:\;}
\newcommand{\myfrac}[2]{\genfrac{}{}{0pt}{}{#1}{#2}}

\newcommand{\1}{\mathbf{1}}

\long\def\symbolfootnote[#1]#2{\begingroup\def\thefootnote{\fnsymbol{footnote}}
\footnote[#1]{#2}\endgroup}

\begin{document}  

\bigskip
\bigskip

\begin{center}
\textsc{\textbf{\Large On random topological Markov chains with big images and preimages}}

\bigskip \bigskip 
{\textsc{ Manuel Stadlbauer}}

\medskip
\textit{\small Departamento de Matem\'atica, Universidade Federal da Bahia, 
Av. Ademar de Barros s/n, 40170-110 Salvador, BA, Brasil
E-mail: {manuel.stadlbauer@ufba.br}\\
{\today}}
\symbolfootnote[0]{ {}\\[-8pt]
Corrected and extended version of the article published as Stochastics and Dynamics 10 (1), 2010, 77-95, DOI: \href{http://dx.doi.org/10.1142/S0219493710002863}{10.1142/S0219493710002863} © World Scientific Publishing Company, \href{http://www.worldscinet.com/sd/}{http://www.worldscinet.com/sd/}.
}

\medskip
\textit{Dedicated to Manfred Denker on the Occasion of His 65th Birthday}\end{center}

\begin{quote} \textit{Abstract.} We introduce a relative notion of the “big images and preimages”-property for random topological Markov chains. This condition then implies that a relative version of the Ruelle-Perron-Frobenius theorem holds with respect to summable and locally Hölder continuous potentials.

\textit{Keywords:} Random countable Markov shift; random bundle transformation; Ruelle– Perron–Frobenius theorem; Markov chains with random transitions; finite primitivity; big images and preimages.

{AMS Subject Classification:}  37D35, 37H99
\end{quote}

\section{Introduction} In this note we give a further contribution to the extension of thermodynamic formalism for topological Markov chains to random transformations and, in particular, obtain a sufficient condition for the existence of random conformal measures and random eigenfunctions of the Ruelle operator which applies e.g. to a random full shift with countably many states. In particular, we obtain an extension of the results for random subshifts of finite type obtained by Bogenschütz, Gundlach and Kifer (\cite{BogenschutzGundlach:1995,Kifer:1996a,Kifer:2008}) to random shift spaces with countably many states. For illustration, we also give applications to countable random matrices, that is, we deduce a Perron-Frobenius theorem and a sufficient condition for the existence of a stationary distribution for a countable-state Markov chain with random transition probabilities.

For deterministic dynamical systems the following results are known. Recall that
it was shown by Sarig (\cite{Sarig:1999}) that the Ruelle-Perron-Frobenius theorem extends to deterministic topological Markov chains with countably many states and locally Hölder continuous potentials if and only if the system is positive recurrent. If the potential is summable, results in this direction are obtained by imposing topological mixing conditions, `finite irreducibility' or `finite primitivity', on the shift space (see \cite{MauldinUrbanski:2001,StratmannUrbanski:2007}). Furthermore, if the topological Markov chain is topologically mixing, then these conditions coincide with the `big images and preimages'-property introduced in \cite{Sarig:2003} where it is shown that this condition is equivalent to positive recurrence for summable potentials. Note that these results are advantageous in many applications since they can be, in contrast to positive recurrence, verified easily. 

The goal of this paper is to obtain an extension of these results to random bundle transformations, that is we consider a commuting diagram (or fibered system)   
\[  \begin{array}{rll}
 X  &\stackrel{T}\longrightarrow  &X     \\
 {\scriptstyle \pi} \downarrow \,& &\downarrow  {\scriptstyle \pi}\\
\Om & \stackrel{\te}\longrightarrow &\Om,
\end{array} \]
where $\te$ is an ergodic automorphism of the abstract probability space $(\Om,P)$ and $\pi$ is onto and measurable. 
With $X_\om$ referring to $\pi^{-1}(\{\om\})$, the  restriction $T_\om: X_\om \to X_{\te \om}$ of $T$ to fibers then has a natural interpretation as a random transformation in a random environment. In here, we consider the class of random topological Markov chains, that is,  $X$ is a subset of $\bbN^\bbN\times \Om$ such that each fiber $X_\om$ has a random Markov structure (for details, see Section \ref{sec:2}).

For the extension of the notion of big images and preimages (b.i.p.) to this setting, we only require that a corresponding property holds for returns to subsets $\Om_{\textrm{\tiny bi}}$ and $\Om_{\textrm{\tiny bp}}$ of positive measure in the base $\Om$. That is, for $\om \in \Om_{\textrm{\tiny bi}}$, there exists a finite union of cylinders $F_{\te\om} \subset X_{\te \om}$ such that $T_\om([a]) \cap  F_{\te\om} \neq \emptyset$ for all cylinders $[a]\subset X_{\om}$ (\emph{big images}) and, for $\om \in \Om_{\textrm{\tiny bp}}$, there exists a finite union of cylinders $F'_{\te^{-1} \om}\subset X_{\te^{-1} \om }$ such that $T_{\te^{-1} \om}(F'_{\te^{-1} \om}) = X_\om$ (\emph{big preimages}), respectively. Note that this property is a purely topological property with respect to the fibers.

We then consider topologically mixing systems equipped with a potential $\phi$ which is locally Hölder continuous in the fibers.
Our further analysis relies on the divergence at the radius of convergence of a random power series whose coefficients are given by random partition functions. Systems with this property will be called of divergence type. As a first result we obtain in Theorem \ref{theo:ergodic} that a  system with summable potential and the b.i.p.-property is of divergence type.

For systems of divergence type with summable potential, it then follows that a random conformal measure exists (Theorem \ref{theo:conformal}). That is, there exists a family of probability measures $\{\mu_\om\}$ and a positive random variable $\la: \Om \to \bbR$ such that, for $x \in X_\om$, 
\[ {d\mu_{\te \om}\circ T_\om}/{d\mu_\om}(x) = \la(\om) e^{P_G(\phi)-\phi(x)}\]
where $P_G(\phi)$ refers to the relative  Gurevi\v{c} pressure as introduced in \cite{DenkerKiferStadlbauer:2008}. The proof of this statement relies on the 
construction of $\la$ as the limit of the quotient of random power series and the application of 
Crauel's random Prohorov theorem (see \cite{Crauel:2002}) to a family of random measures. Note that 
the construction of this family of random measures is an adaption of the construction in \cite{DenkerKiferStadlbauer:2008}. However, it turns out that the summability assumption significantly simplifies the tightness argument compared to the proof in there.
   
In particular, this result gives that $\la e^{P_G(\phi)}$ is the spectral radius of the dual of the random Ruelle operator. For systems with the b.i.p.-property, the identification of $\la$ as  quotient of random power series then gives rise to application of results in \cite{DenkerKiferStadlbauer:2008}, that is the system is positive recurrent and a relative version of the Ruelle-Perron-Frobenius theorem holds (Corollary \ref{cor:pr} and Theorem \ref{theo:rpf}). As immediate consequences of these results, we obtain a Perron-Frobenius theorem for random matrices (Corollary \ref{cor:pf}) and an application to random stochastic matrices (Corollary \ref{theo:stochastic:matrix}). 

\bigskip
\setlength{\fboxsep}{8pt}
\setlength{\fboxrule}{1pt}
\noindent\fbox{\parbox{0.952\textwidth}{This note is an extended and corrected version of \cite{Stadlbauer:2010}. It now contains a correct statement and extended proof of theorem \ref{theo:conformal}, a remark on the left-out statement in theorem \ref{theo:conformal} (remark  \ref{rem:missing_statement}), and an alternative, simplified proof of corollary \ref{cor:pr} which is independent from that statement. Furthermore, the line before corollary \ref{cor:pr} was deleted.}}

\section{Preliminaries} \label{sec:2}

Let  $\te$ be an  automorphism (i.e. bimeasurable, invertible and probability preserving) of the probability space $(\Om,\cF,P)$, $\ell=\ell_\om>1$ be a  $\bbN\cup\{\infty\}$-valued random variable and, for a.e. $\om \in \Om$, let   
$A_\om=\big(\al_{ij}(\om),\, i < \ell_\om,j < \ell_{\te\om}\big)$ be a matrix
with entries $\alpha_{ij}(\om)\in\{ 0,1\}$ such that $\om \mapsto A_\om$ is measurable and $\sum_{j < \ell_{\te\om}}{a_{ij}(\om)}>0$ for all $i < \ell_\om$.
For the random shift spaces 
\[
X_\om=\{ x=(x_0,x_1,...):\, \al_{x_ix_{i+1}}(\te^i\om)=1\,\,\forall i=0,1,...
\},
\]
the (random) shift map $T_\om :X_\om\to X_{\te\om}$ is defined by $T_\om :(x_0,x_1,x_2...)=
(x_1,x_2,...)$. This gives rise to a globally defined map $T$ of $X$ where $X:=\{(\om,x):\, x\in X_\om\}$ and  $T(\om,x)=(\te\om,T_\om x)$. In this situation, the pair $(X,T)$ is referred to as a \emph{random countable topological Markov chain}.
For $n \in \bbN$, set $T^n_\om=T_{\te^{n-1}\om}\circ \cdots \circ T_{\te\om} \circ T_\om$ and note that $T^n(\om,x)=(\te^n\om,T^n_\om x)$.

A finite word $a=(x_0, x_1, \ldots,x_{n-1}) \in \bbN^n$ of length $n$ is called \emph{$\om$-admissible}, if $x_i < \ell_{\te^i \om}$ and $\al_{x_ix_{i+1}}(\te^i\om)=1$, for $i=0, \ldots , n-1$. In here, $\cW^n_\om$ denotes the set  of $\om$-admissible words of length $n$ (in particular, $\cW^1_\om= \{ a:a<\ell_\om\}$) and, for $a=(a_0, a_1, \ldots,a_{n-1}) \in \bbN^n$,  
\[[a]_\om = [a_0,a_1,...,a_{n-1}]_\om :=\{ x\in X_\om:\, x_i=a_i,\, i=0,1,...,n-1\}\]
is called \emph{cylinder set}. The set of those $\om \in \Om$ where the cylinder is nonempty will be denoted by $\Om_a$, that is 
\[\Om_a =\{\om \with [a]_\om\ne\emptyset\} = \{\om: a \in \cW^n_\om\}.\]
Finally, $\cW^n$ refers to the set of words $a$ of length $n$ with $P(\Om_a)>0$.
In this paper, we exclusively consider \emph{topologically mixing} random topological Markov chains. That is, for 
$a,b\in\cW^1$, there exists a $\bbN$-valued random variable $N_{ab}=
N_{ab}(\om)$ such that, for $n\geq N_{ab}(\om)$, $a \leq \cW^1_\om$ and
$\te^n\om \in \Om_b$, we have that $[a]_\om\cap(T^n_\om)^{-1}[b]_{\te^n\om}\ne 
\emptyset$.

As mentioned above we are interested in thermodynamic aspects of random topological Markov chains with respect to locally  Hölder continuous potentials. Therefore, recall that, for a function $\phi:\, X\to\bbR$, $(\om,x) \mapsto \phi^\om(x)$, the $n$-th variation is defined by
\[V^\om_n(\phi)=\sup\{ |\phi^\om(x)-\phi^\om(y)|:\, x_i=y_i,\, i=0,1,\ldots,
n-1\}.\]
The function $\phi$ is referred to as a \emph{locally fiber H\" older continuous function with index $k\in \bbN$} if there exists
a random variable $\ka=\ka(\om)\geq 1$ such that $\int\log\ka dP<\infty$ and,
for all $n\geq k$, $V^\om_n(\phi)\leq\ka(\om)r^n$. For abbreviation, 
such a function will be referred to as a \emph{$k$-Hölder continuous} function. This then leads to the following elementary but useful estimate. For $n \leq m$, $x,y \in [a]_\om$ for some $a \in \cW_\om^m$, and a $(m-n+1)$-Hölder continuous function $\phi$,  
\begin{align*}
|\phi_n^\om (x)-\phi_n^\om (y)| & \leq \sum_{k=0}^{n-1}| \phi^{\te^k\om}(T_\om^k(x)) - \phi^{\te^k\om}(T_\om^k(y))| \leq \sum_{k=0}^{n-1} V_{m-k}^{\te^k \om}(\phi)\\
  & \leq \sum_{k=0}^{n-1} \ka(\te^k \om)r^{m-k}  
 \leq \sum_{k>m-n} \ka(\te^{m-k}\om) r^{k} = r^{m-n}\sum_{k=1}^\infty \ka(\te^{n-k}\om) r^k. 
\end{align*}
Since $\log \ka \in L^1(P)$, we obtain $(\log\ka)/n \to 0$ as a consequence of the ergodic theorem. So the radius of convergence of the series on the right hand side of the above estimate is equal to 1 and, in particular, the right hand side  is finite. 
 Hence,  for a locally fiber H\" older continuous function with index less than or equal to $(m-n+1)$,  
\begin{equation}\label{eq:key} (B_{\te^n \om})^{-1} \leq  (B_{\te^n \om})^{-r^{m-n}} \leq e^{\phi_n^\om (x)-\phi_n^\om (y)} \leq  (B_{\te^n \om})^{r^{m-n}} \leq  B_{\te^n \om}\end{equation}
where $B_\om:= \exp \sum_{k=1}^\infty \kappa(\te^{-k}\om)r^k$. {Note that this definition differs from the one in \cite{DenkerKiferStadlbauer:2008} by the choice of the element in the base - in here we replaced $\om$ by $\te^n\om$.} 
A further basic notion is the \emph{(random) Ruelle operator} $L_\phi$ associated to a potential (function) $\phi =(\phi^\om):X \to \bbR$ which is defined by, for a function $f:X \to \bbR$,
\[
L^\om_\phi f(\te\om, x)=\sum_{y\in X_\om,T_\om y=x}e^{\phi^\om(y)}f(\om,y).
\]

In this note, we consider potentials $\phi$ satisfying some of the following additional assumptions.
\begin{itemize}
  \item[\textrm{(H1)}] The potential $\phi$ is $1$-H\" older continuous and $\int\log B_\om dP(\om) < \infty$.  
  \item[\textrm{(H2)}] The potential $\phi$ is $2$-H\" older continuous and $\int\log B_\om dP(\om) < \infty$. 
  \item[\textrm{(S1)}] 
  $\int \log M_\om dP(\om)<\infty$ where $M_\om :=  \sup \{L^\om_\phi(1)(x):x \in X_{\te\om}\}$. 
  \item[\textrm{(S2)}] 
  $\int \log m_\om  dP(\om)> - \infty$ where  $m_\om:= \inf \{L^\om_\phi(1)(x):x \in X_{\te\om}\}$. 
\end{itemize}
These assumptions might be seen as randomized versions of Hölder continuous (H1-2) and summable potentials (S1-2), respectively. 
Recall that, if $\phi$ is locally fiber H\" older continuous and $\ka$ is integrable, then (H1) holds (see \cite{DenkerKiferStadlbauer:2008}).  Also note that (S1-2) is equivalent to  $\|\log L^\om_\phi(1)\| \in L^1(P)$. Below, after introducing big images and preimages, we will give a further Hölder condition (H$^\ast$)  for which (H2) holds and  $1$-H\" older continuity is only required on a subset of $\Om$.

\section{Partition functions and big images and preimages} \label{sec:3}

In this section, we introduce the notion of big images and preimages for random topological Markov chains. Moreover, we discuss immediate consequences in terms of estimates for the random version of the Gurevi\v{c} partition functions. These estimates will then be used to prove that the preimage function diverges at its radius of convergence (Theorem \ref{theo:ergodic}).

In order to define the relevant objects, we introduce the following notation. For $a,b \in \cW^1$, $\om \in \Om_a$ and $n \in \bbN$, set 
\begin{align*} \cW_\om^n(a,b) := \{(w_0, \ldots, w_{n-1})\in \cW^n_\om \with w_0=a, w_{n-1}b \in \cW^{2}_{\te^{n-1}\om}\}.\end{align*}
Moreover, for $w \in \cW^n_\om$, and $m \leq n$, set $\exp(\phi_m^\om([w])):= \sup\{\exp(\phi_m^\om(x))\with x \in [w]_\om\}$. As an immediate consequence of (\ref{eq:key}) we have, for a $k$-Hölder continuous potential $\phi$, 
\begin{equation}\label{eq:finite} 0 < \inf\{\exp(\phi_{n-k+1}^\om(x))\with x \in [w]_\om\} \leq \exp(\phi_{n-k+1}^\om([w]))< \infty \hbox{ a.s.}\end{equation} 

We now consider a fixed topologically mixing random Markov chain $(X,T)$, a potential $\phi$ satisfying (H2)  and $a \in \cW^1$. For $\om \in \Om_a$ and $n \in \bbN$, the \emph{$n$-th (random) Gurevi\v{c} partition function} is defined by 
\[Z_n^\om(a) := \sum_{w \in \cW_\om^n(a,a)} e^{\phi_{n}^\om([wa])}, \]
where we use the convention that  $Z_n^\om(a)=0$ if $ \cW_\om^n(a,a) = \emptyset$. 
Note that this definition differs from the one in \cite{DenkerKiferStadlbauer:2008}. In here, $\phi_{n}^\om([w])$ is replaced by $\phi_{n}^\om([wa])$ in order to obtain a partition function applicable to $2$-Hölder continuous potentials.
Since $(X,T)$ is topologically mixing, it follows that $Z_n^\om(a)>0$ for all $n \geq N_{aa}(\om)$ with $\te^n\om \in \Om_a$.  
Furthermore, given a measurable family $\{\xi_\om \in [a]_\om : \om \in \Om\}$, the \emph{$n$-th local preimage function} is defined by 
\[\cZ_n^\om(a) := \sum_{w \in \cW_\om^n(a,a)} e^{\phi_{n}^\om(\tau_w(\xi_{\te^n\om}))} = L_\phi^{\om,n}(1_{[a]})(\xi_{\te^n\om})\] 
where $\tau_w$ refers to the inverse branch $T_\om^n([w]_\om) \to [w]_\om$. In particular, if $\phi$ is $2$-Hölder, then   (\ref{eq:key}) implies that $Z_n^\om(a) \geq \cZ_n^\om(a) \geq Z_n^\om(a) B_{\te^n\om}^{-1}$. 
Moreover, the  \emph{$n$-th preimage function} is defined by, for $\om \in \Om$,
\[\cZ_n^\om := \sum_{w \in \cW_\om^n} e^{\phi_{n}^\om(\tau_w(\xi_{\te^n\om}))} = L_\phi^{\om,n}(1)(\xi_{\te^n\om}).\]
As a consequence of (S1), we have $\cZ_n^\om \leq M_\om \cdots M_{\te^{n-1}\om}< \infty$.
 Finally, set    
\[A_n^\om := \sum_{w \in \cW_\om^n} e^{\phi_{n}^\om([w])}\]
and note that $0 < A_n^\om \leq \infty$. We now introduce 
the \emph{relative Gurevi\v{c} pressure} $P_G(\phi)$ adapted to the situation under consideration. 
 For $\Om' \subset \Om$ and $\om \in \Om$, set $J_\om({\Om'}) := \{ n \in \bbN \with \te^n\om \in \Om' \}$ and choose  $N\in \bbN$ such that $\Om^\ast:= \{\om \in \Om_a \with N_{aa}(\om) \leq N \}$ is a set of positive measure. The following proposition is a slight generalization of Theorem 3.2 in \cite{DenkerKiferStadlbauer:2008} to $2$-Hölder continuous (H2) and summable (S1) potentials. 
 
\begin{proposition}\label{prop:pressure} For a mixing system $(X,T)$ and a potential satisfying (H2) and (S1), the limits
  \[ P_G(\phi):= \lim_{\myfrac{n \to \infty,}{n \in J_\om({\Om^\ast})}} \frac{1}{n} \log Z_n^\om(a) =  \lim_{\myfrac{n \to \infty,}{n \in J_\om({\Om^\ast})}} \frac{1}{n} \log \cZ_n^\om(a) \geq -\infty\]
exist, are a.s. constant with respect to $\om$ and independent of the choice of $a$ and $N$.
\end{proposition}

\begin{proof} Since most of the arguments can be found in \cite{DenkerKiferStadlbauer:2008}, we only give a sketch of proof. For a.e. $\om \in \Om^\ast$ and $m,n \geq N$ with $\te^m\om, \te^{m+n}\om  \in \Om^\ast$, it follows from (\ref{eq:key}) that
\begin{equation}\label{eq:submult_CZ} \cZ^\om_m(a)\cZ^{\te^m \om}_n(a) \leq B_{\te^m \om} \cZ^\om_{m+n}(a).\end{equation}
It is well known that the induced transformation ${\hat\te}: \Om^\ast \to \Om^\ast$ given by 
\begin{align*}
\eta:&\;  \Om' \to \bbN, \quad \om \to \eta(\om) := \min \{n \in \bbN\with \te^n \om \in \Om'\}\\
{\hat\te}: &\; \Om' \to \Om', \quad \om \to \te^{\eta(\om)}\om.
\end{align*} 
is an invertible, measure preserving, conservative and ergodic transformation with respect to $P$ restricted to $\Om^\ast$. Set $\eta_k(\om) := \sum_{l=0}^{k-1} \eta({\hat\te}^l\om)$.
It then follows from (\ref{eq:submult_CZ}), for $M\geq N$ and $k,l \in \bbN$, that 
\[-\log \cZ^\om_{\eta_{(k+l)M}(\om)}(a) + \log \cZ^\om_{\eta_{kM}(\om)}(a) + \log \cZ^{{\hat\te}^{kM}\om}_{\eta_{lM}({\hat\te}^{kM}\om)}(a)  \leq \log B_{{\hat\te}^{kM}\om}. \]
Since $\cZ^\om_n(a) \leq \cZ^\om_n \leq M_{\om}M_{\te{\om}} \cdots M_{\te^{n-1}\om}$, it follows from 
(H2), (S1) and Kac's theorem that the almost subadditive ergodic theorem as stated in \cite{Derriennic:1983} is applicable to $-\log \cZ_\cdot^\om(a)$  with respect to the measure preserving transformation ${\hat\te}^M$. Since the quotient $\eta_k/k$ converges by Birkhoff's ergodic theorem, it follows that 
\[f(\om):= \lim_{k \to \infty} \frac{1}{\eta_{kM}(\om)} \log \cZ^\om_{\eta_{kM}(\om)}(a) \]
exists a.s. and is ${\hat\te}^M$-invariant.
 It is now easy to see that 
this limit is independent from the choice of $M\geq N$ and hence the limit is a constant function.
In order to show that the limit does exist along $J_\om({\Om^\ast})$, we now use a different argument as in \cite{DenkerKiferStadlbauer:2008}. 
For $k > 3N$, set $a_k:= N + (k \mod N)$. Then $k - a_k$ is a multiple of $N$ and $2N > a_k \geq N$. In particular, $\cZ^\om_{\eta_{a_k}(\om)}(a), \cZ^{{\hat\te}^{a_k}(\om)}_{\eta_{k-a_k}(\om)}(a)>0$.  Hence, by (\ref{eq:submult_CZ}),  
\[\frac{1}{\eta_k(\om)} \log \left( \cZ^\om_{\eta_{a_k}(\om)}(a) \cZ^{{\hat\te}^{a_k}(\om)}_{\eta_{k-a_k}(\om)}(a)  \right) \leq \frac{1}{\eta_k(\om)} \log \left( B_{{\hat\te}^{a_k}(\om)} \cZ^\om_{\eta_k}(a) \right) .  \]
By passing to the limit, we obtain that $\inf \{(\log \cZ^\om_n(a))\with n \in  J_\om({\Om^\ast})\} \geq f(\om)$ a.s. The other direction then follows by the same argument with $b_k = 2N - (k \mod N)$ and using
\[\cZ^{{\hat\te}^{-b_k}(\om)}_{\eta_{b_k}({\hat\te}^{-b_k}(\om))}(a) \cZ^{\om}_{\eta_{k-b_k}(\om)}(a)
 \leq B_\om  \cZ^{{\hat\te}^{-b_k}(\om)}_{\eta_{k}}(a). \]
The remaining assertions now follow from $Z_n^\om(a) \geq \cZ_n^\om(a) \geq Z_n^\om(a) B_{\te^n\om}^{-1}$ and Step 2 in the proof of Theorem 3.2 in \cite{DenkerKiferStadlbauer:2008}.
\end{proof}

We now introduce the notion of big images and preimages. In here, we will write 
$\# B$ for the cardinality of a set $B$. 
So assume that there exists $\Om_{\textrm{\tiny bi}}\subset \Om$ of positive measure and a family $\{\cI_{\textrm{\tiny bi}}^\om \subset \cW^1_{\om} : \om \in \Om_{\textrm{\tiny bi}}\}$ such that 
\begin{enumerate}
  \item $\#\cI_{\textrm{\tiny bi}}^\om < \infty$,
  \item for each $a \in \cW^1_{\te^{-1}\om}$, there exists $b \in \cI_{\textrm{\tiny bi}}^\om$ with $ab \in \cW^2_{\te^{-1}\om}$.
\end{enumerate}
We then say that $(X,T)$ has the \emph{big image property}. By choosing a subset of $\Om_{\textrm{\tiny bi}}$, one may assume without loss of generality that there exists a finite set $\cI_{\textrm{\tiny bi}}$ such that $\cI_{\textrm{\tiny bi}}^\om \subset \cI_{\textrm{\tiny bi}}$ for each $\om \in \Om_{\textrm{\tiny bi}}$.

Moreover, if there exists $\Om_{\textrm{\tiny bp}}\subset \Om$ of positive measure, and a family $\{\cI_{\textrm{\tiny bp}}^\om \subset \cW^1_{\te^{-1}\om} : \om \in \Om_{\textrm{\tiny bp}}\}$ such that
\begin{enumerate}
  \item $\# \cI_{\textrm{\tiny bp}}^\om < \infty$,
  \item for each $a \in \cW^1_\om$, there exists $b \in \cI_{\textrm{\tiny bp}}^\om$ with $ba \in \cW^2_{\te^{-1}\om}$,
\end{enumerate}
then $(X,T)$ is said to have the \emph{big preimage property}. As above, one may assume without loss of generality that each $\cI_{\textrm{\tiny bp}}^\om$ is a subset of a globally defined finite set $\cI_{\textrm{\tiny bp}}$. If $(X,T)$ is topologically mixing and has the {big image and big preimage property}, then $(X,T)$ is said to have the  \emph{(relative) b.i.p.-property}. 
\begin{lemma} \label{lem:top_exact} If $(X,T)$ has the b.i.p.-property, then, for $a \in \cW^1$ and almost every $\om \in \Om_a$, there exist $\alpha_\om, \beta_\om \in \bbN$ such that  
\begin{enumerate}
  \item $\cW^n_\om(a,b) \neq \emptyset$, for all $n \geq \alpha_\om$ and $b \in \cW^1_{\te^n \om}$, 
  \item $\cW^n_{\te^{-n}\om}(b,a) \neq \emptyset$, for all $n \geq \beta_\om$ and $b \in \cW^1_{\te^{-n}\om}$. 
\end{enumerate}
\end{lemma}
\begin{proof} In order to show the first assertion, set $N_\om:= \max\{ N_{ac}(\om) \with c \in \cI_{\textrm{\tiny bp}}\}$, and $\alpha_\om:= \min\{n \geq N_\om : \te^n \in \Om_{\textrm{\tiny bp}}\}$. The second assertion follows by a similar construction. 
\end{proof}

The following Lemma now shows that the above partition and preimage functions are proportional to each other along subsequences. In the proof  we only require that $\phi^\om$ is 1-Hölder for $\om \in \te^{-1}(\Om_{\textrm{\tiny bi}} \cup \Om_{\textrm{\tiny bp}}) $. The precise condition is as follows.
\begin{itemize}
  \item[{(H$^\ast$)}] The potential $\phi$ has property (H2) and, for a.e. $\om \in \te^{-1}(\Om_{\textrm{\tiny bi}} \cup \Om_{\textrm{\tiny bp}})$, we have $V_1^\om(\phi) < \infty$. 
\end{itemize}
\begin{lemma} \label{lem:partition_functions} For $(X,T,\phi)$ with the b.i.p.-property, (H$^\ast$) and (S1-2), the following holds.
\begin{enumerate}
  \item  For  a.e. $\om \in \Om_a$ and $k,n \in \bbN$ with $k \geq \alpha_{\om}$, $\te^k\om \in \Om_{\textrm{\tiny bp}}$ and $\te^{k+n}\om \in \Om_a$, there  exists $1\leq C_\om({a,k})<\infty$ such that
  \[\cZ_{n}^{\te^k\om} \leq  C_{\om}({a,k})\; \cZ_{k+n}^{\om}(a).\]
  \item For  a.e. $\om \in \Om$ and $n,k \in \bbN$ with $\te^n \om \in \Om_{\textrm{\tiny bi}}$, $\te^{k+n}\om \in \Om_a$ and $k \geq \beta_{\te^n\om}$, there exists $1\leq D_{\te^n\om}({a,k})<\infty$ such that
  \[ A_n^{\om} \leq  B_{\te^n\om} D_{\te^n\om}({a,k})^{-1} \; \cZ_{n+k}^{\om}.\]
\end{enumerate}
Moreover, $P_G(\phi)$ is finite, and $C_\om(a,k)$ and $D_\om(a,k)$ are measurable.
\end{lemma}
\begin{proof} 
In order to show the first assertion, note that the big preimage property combined with Lemma \ref{lem:top_exact} implies the existence of $\{v_j \in  \cW^k_{\om} \with j= 1, \ldots, \# \cI_{\textrm{\tiny bp}}^{\te^k\om} \}$ such that, for each $b \in \cW^1_{\te^k\om}$, there exists $j \in \{1, \ldots, \# \cI_{\textrm{\tiny bp}}^{\te^k\om}\}$ with $v_j b\in \cW^{k+1}_{\om}$. This then gives that  
\begin{align*}
\cZ_{k+n}^{\om}(a) & \geq \sum_{j=1}^{\# \cI_{\textrm{\tiny bp}}^{\te^k\om}} \sum_{w \with v_j w\in \cW^{k+n}_{\om}} e^{\phi^{\om}_{n+k}(\tau_{v_jw}(\xi_{\te^{k+n} \om}))}\\
& \geq \sum_{j=1}^{\# \cI_{\textrm{\tiny bp}}^{\te^k\om}} \inf\left\{e^{\phi^{\om}_{k}(x)} : x \in [v_j]_{\om} \right\}  \sum_{w \with v_j w\in \cW^{k+n}_{\om}} e^{\phi^{\te^k\om}_{n}(\tau_{w}(\xi_{\te^{k+n} \om}))}\\
& \geq \inf \left\{e^{\phi^{\om}_{k}(x)} : x \in [v_j]_{\om}, j = 1, \ldots,  {\# \cI_{\textrm{\tiny bp}}^\om}\right\} \; \cZ_n^\om =: (C_{\om}({a,k}))^{-1}\; \cZ_{n}^{\om}.
\end{align*}
Observe that $C_\om(a,k)>0$ which follows from (H$^\ast$) and  (\ref{eq:finite}).
Moreover, by choosing e.g. $v_1, \ldots v_{\# \cI_{\textrm{\tiny bp}}^\om}$ to be minimal with respect to the lexicographic ordering, it follows that $\om \to C_\om(a,k)$ is measurable.
Assertion (ii) follows by a similar argument, that is by
\begin{align*} \cZ_{n+k}^{\om} & \geq  \sum_{j=1}^{\# \cI_{\textrm{\tiny bi}}^{\te^n\om}} \sum_{w \with  wv_j \in \cW^{n+k}_{\om}} e^{\phi^{\om}_{n+k}(\tau_{w,v_j}(\xi_{\te^{n+k} \om}))} \\
& \geq A_n^\om B^{-1}_{\te^n\om} \exp({-V_1^{\te^{n-1}\om}(\phi)} )
\sum_{j=1}^{\# \cI_{\textrm{\tiny bi}}^{\te^n\om}} \inf\left\{e^{\phi^{\te^n\om}_{k}(x)} : x \in [v_j]_{\te^n\om} \right\} ,
\end{align*}
where $\{v_j \in  \cW^k_{\te^n\om} \with j= 1, \ldots, \# \cI_{\textrm{\tiny bi}}^{\te^n\om} \}$ are constructed from the big image property. For the proof of $|P_G(\phi)| < \infty$, note that 
\[ \frac{1}{n}\sum_{k=0}^{n-1} \log m_{\te^k \om} \stackrel{n\to \infty}{\longrightarrow} \int \log m_\om dP \;\hbox{ and } \;\frac{1}{n}\sum_{k=0}^{n-1} \log M_{\te^k \om} \stackrel{n\to \infty}{\longrightarrow} \int \log M_\om dP  \]
by the ergodic theorem. It hence follows from $\cZ^\om_n(a) \leq M_\om \cdots M_{\te^{n-1}\om}$ that 
$P_G(\phi)<\infty$. Furthermore, from assertions (i) and (ii) combined  
with $\log A_n^\om \geq \sum_{k=0}^{n-1} \log m_{\te^k \om}$ and the convergence in  Proposition \ref{prop:pressure}, we obtain that $P_G(\phi)\geq \int \log m_\om dP(\om)$. 
\end{proof}

Using these estimates, we are now in position to prove the main result of this section. In the statement of the theorem, $\Om^\ast$ refers to the subset of $\Om_a$ in the definition of the relative Gurevi\v{c} pressure.
\begin{theorem}\label{theo:ergodic}
Assume that $(X,T)$ has the b.i.p.-property and (H$^\ast$) and (S1-S2) are satisfied. Then $P_G(\phi)$ is finite and, for a.e. $\om \in \Om$,   
\[ \sum_{n \in J_\om(\Om^\ast)} s^n \cZ_n^\om \quad \begin{cases} < \infty  &:\; s< e^{-P_G(\phi)},\\ =\infty &:\; s= e^{-P_G(\phi)}.\end{cases}\]
\end{theorem}

\begin{proof} Note that $P_G(\phi)$ is finite by Lemma \ref{lem:partition_functions}.
By replacing $\phi$ by $\phi-\log \phi$ we now assume without loss of generality that $P_G(\phi)=0$. From Lemma \ref{lem:partition_functions} (i), it then follows, for a.e. $\om \in \Om_{\textrm{\tiny bp}}$, that 
\[\lim_{n \in J_\om (\Om^\ast)} \frac{1}{n} \log \cZ_n^\om =0.\]
We will show that $\sum A_n^\om < \infty$ leads to a contradiction of  $P_G(\phi)=0$. So assume that, for a.e. $\om \in \Om_{\textrm{\tiny bi}}$,
$ \sum_{n \in J_\om(\Om_{\textrm{\tiny bi}})}  A_n^\om < \infty$.
Hence, for $\epsilon >0$, there exist $\Om' \subset \Om_{\textrm{\tiny bi}}$ and $N \in \bbN$ such that $A_n^\om < \epsilon$ for all $\om \in \Om'$ and $n \geq N$. Now consider the jump transformation $\te^\ast: \Om' \to \Om'$ given by 
\begin{align*}
\eta^\ast:&\;  \Om' \to \bbN, \quad \om \to \eta^\ast(\om) := \min \{n \in \bbN\with n \geq N, \te^n \om \in \Om'\},\\
\te^\ast: &\; \Om' \to \Om', \quad \om \to \te^{\eta^\ast(\om)}\om.
\end{align*} 
Note that $\te^*$ is invertible, and that $P|_{\Om'}$ is a finite $\te^*$-invariant measure (see e.g. \cite{Thaler:1983}). In particular, it follows that $\te^*(\Om')= \Om' \mod P$. Set
\[ \eta_k^\ast(\om):= \sum_{i=0}^{k-1} \eta^\ast({(\te^\ast)^i \om}).\]
Since the sequence $(\log A_n^\om)$ is subadditive, it follows that $ A_{\eta^\ast_k(\om)}^\om \leq \epsilon^k$. Furthermore, by the ergodic theorem, $\eta_k^\ast(\om)/k$ converges to an invariant function which is bigger than or equal to $N$. In particular,  
\[ \lim_{k \to \infty} \frac{1}{\eta_k^\ast(\om)}  \log A_{\eta_k^\ast(\om)}^\om \leq \frac{k}{\eta_k^\ast(\om)} \log \epsilon \leq \frac{\log \epsilon}{N} \hbox{ a.s.}\]
Using  $A_n^\om \geq \cZ_n^\om$, we obtain that $\lim_{n \in J_\om(\widetilde\Om)} (\log  \cZ_n^\om)/n <0 $ for a.e. $\om \in \Om_{\textrm{\tiny bi}}$ and a suitable subset $\widetilde\Om \subset \Om^\ast$ of positive measure. Since this is a contradiction to $P_G(\phi)=0$, it follows that 
\[  \sum_{n \in J_\om(\Om_{\textrm{\tiny bi}})}  A_n^\om = \infty\]
for a.e. $\om \in \Om_{\textrm{\tiny bi}}$ and, by subadditivity, for a.e.  $\om \in \Om$. The assertion then follows from Lemma \ref{lem:partition_functions} (ii).
\end{proof}

\section{Random eigenvalues and conformal measures} \label{sec:4}
The first step in this section is to construct random eigenvalues and conformal measures for random topological Markov chains for which the sum of the preimage function diverges. As a corollary, we obtain that the random eigenvalue can be identified with the quotient of two random power series. In particular, this then gives in analogy to deterministic topological Markov chains (see \cite{Sarig:2003}) that the b.i.p.-property implies positive recurrence. Throughout this section we assume that $(X,T)$ is topologically mixing, $\phi$ satisfies (H2) and (S1-2), and $P_G(\phi)$ is finite. Hence, we may assume without loss of generality that $P_G(\phi)=0$. Now fix $a \in \cW^1$ and, for $\widetilde \Om \subset \Om_a$, $\om \in \Om$ and $0<s\leq 1$, set
\[ P_\om(s):= \sum_{n \in J_\om(\widetilde \Om)} s^n \cZ_n^\om.\]
If there exists $\widetilde \Om \subset \Om_a$ such that $P_\om(1)=\infty$ and $P_\om(s)< \infty$ for $0<s<1$, we say that 
$(X,T,\phi)$ is of \emph{divergence type}. 
In particular, observe that for a system of divergence type, we have $\lim_{n \in J_\om(\widetilde \Om)}(\log \cZ_n^\om)/n = 0 = P_G(\phi)$ by Hadamard's formula for the radius of convergence. 
Also note that systems with the b.i.p.-property are in this class as a consequence of Theorem \ref{theo:ergodic}.
\begin{lemma}\label{lem:weak_limit}
 There exists a sequence $(s_n:n \in \bbN)$ with $s_n \nearrow 1$ and $\la^\ast :\om \to \bbR$ with $\log \la^\ast \in L^1(P)$ such that 
\[ \int g(\om) \log (\la^\ast(\om)) dP(\om) = \lim_{n \to \infty} \int g(\om) \log(P_\om(s_n)/P_{\te \om(s_n)})dP(\om) \]
for all $g \in L^\infty(P)$. Furthermore, we have $\int \log \la^\ast dP = 0$ and $ 
m_\om \leq \la^\ast(\om) \leq M_\om$, for $P$-a.e. $\om\in \Om$.
\end{lemma}
\begin{proof} 
 Observe that  
\begin{align*}
P_\om(s)& = \sum_{n \in J_\om(\widetilde \Om)} s^n \sum_{x \in X_\om: T^n_\om(x)= \xi_{\te^n\om}} e^{\phi^\om(x)} e^{\phi_{n-1}^{\te\om}(T_\om(x))} \\ 
&=
sL^\om_\phi(1)(\xi_{\te\om}) +  \sum_{n \in J_\om(\widetilde \Om), n \geq 2} s^n \sum_{y \in X_{\te\om}: T^{n-1}_{\te \om}(y)= \xi_{\te^n\om}} L^\om_\phi(1)(y) e^{\phi_{n-1}^{\te\om}(y)} \\
&\leq s \cdot M_\om (1+ P_{\te\om}(s)).
 \end{align*}
By applying the same argument to obtain the lower bound and using $\left(1 + ({P_{\te \om}(s)})^{-1}\right) \geq 1$, we arrive at 
\begin{equation} \label{eq:estimate_quotient} 
s  m_\om \leq \frac{P_\om(s)}{P_{\te \om}(s)} \leq  s M_\om\left(1 + ({P_{\te \om}(s)})^{-1}\right)  . \end{equation} 
Since $\log \| L^\om_\phi(1) \|\in L^1(P)$, the set $\{\log(P_{\om}(s)/P_{\te\om}(s)): s < 1 \}$ is uniformly integrable. This shows the existence of $\log \la^\ast \in L^1(P)$ as a weak limit. By applying the ergodic theorem, it then follows that $\int \log \la^\ast dP =0$. 
 The remaining assertion can be proved by combining $\lim_{s \to 1+} P_{\te \om}(s) = \infty$  with the above inequalities. 
\end{proof}

In order to obtain pointwise convergence of $P_\om(s)/P_{\te \om}(s)$ as $s \to 1$, we  construct a random conformal measure using a randomized version of the construction in \cite{DenkerUrbanski:1991b}. As a consequence of (S1-2), the construction  and the proof of relative tightness will turn out to be significantly easier than in \cite{DenkerKiferStadlbauer:2008}. 
For $s<1$ and $\om \in \Om$, set 
\[\mu_{\om,s} := \frac{1}{P_\om(s)} \sum_{n \in J_\om(\widetilde \Om)} s^n \sum_{x: T_\om^n(x)= \xi_{\te^n\om}}e^{\phi_n^\om(x)} \delta_x \]
where $\delta_x$ refers to the Dirac measure at $x \in X_\om$. For $A \in \cB_\om$, it  hence follows that
\[\mu_{\om,s}(A) := \frac{1}{P_\om(s)} \sum_{n \in J_\om(\widetilde \Om)} s^n L_\phi^{\om,n}(1_A)(\xi_{\te^n\om}).\]
In order to show that a reasonable limit of this family of measures exists (for $s \nearrow 1$), we will employ Crauel's random Prohorov theorem (see \cite{Crauel:2002}). So recall that $\{\mu_{\om,s}: \om \in \Om, s\geq s_0\}$ is relatively tight if for all $\epsilon >0$ there exists a set $K \subset X$ such that $K\cap X_\om$ is compact for a.e. $\om \in \Om$ and $\int \mu_s(K)dP > 1- \epsilon$ for all  $s>s_0$. 
\begin{lemma}\label{lem:tight} The family $\{\mu_{\om,s}: \om \in \Om, n \in \bbN\}$ is relatively tight. 
\end{lemma}
\begin{proof} For  the proof, for $k \in \bbN$, $\om  \in \Om$ and $b \in \bbN$, set
\begin{align*}
A_\om^{k,b} &:= \{ (x_0,x_1, \ldots) \in X_\om \with  x_k = b\} = T_\om^{-k}([b]_{\te^k\om}),\\
E_n^\om & := T_\om^{-n}(\{\xi_{\te^n \om}\}), \quad E_n^\om(b,k) :=  E_n^\om \cap T_\om^{-k}([b]_{\te^k\om}).
\end{align*}
By construction, it then follows that 
\begin{align*}
&\mu_{\om,s}(A_\om^{k,b}) \\
  = & \frac{1}{P_\om(s)} \left(
  \sum_{\genfrac{}{}{0pt}{}{n \in J_\om(\widetilde{\Om}), n \leq k,}{x \in E_n^\om(b,k)}} s^n e^{\phi_n^\om(x)}  + 
 \sum_{x \in E_{k+1}^\om(b,k)} s^{k+1} e^{\phi_n^\om(x)} + 
 \sum_{\genfrac{}{}{0pt}{}{n \in J_\om(\widetilde{\Om}), n \geq k+2,}{x \in E_n^\om(b,k)}} s^n e^{\phi_n^\om(x)} \right)\\
 =: &\frac{1}{P_\om(s)} \left( \Sigma_1^\om(b) + \Sigma_2^\om(b) + \Sigma_3^\om(b)\right). 
\end{align*}
For the third summand, one immediately obtains that
\begin{align*}
\frac{\Sigma_3^\om(b)}{P_\om(s)} & \leq \frac{s^{k+1}}{P_\om(s)} \sum_{n \in J_\om(\widetilde{\Om}), n \geq k+2} s^{n-(k+1)} \\
& \phantom{\leq} \cdot \sum_{w \in \cW_\om^{k}\with wb \in \cW_\om^{k+1}} e^{\phi_k^\om([wb])}   \sum_{x \in E_{n-(k+1)}^{\te^{k+1}\om} \cap T_{\te^k\om}([b]_{\te^k\om})}   e^{\phi_{n-(k+1)}^{\te^{k+1}\om}(x)}\\
 & \leq   \frac{s^{k+1} P_{\te^{k+1}\om}(s)}{P_\om(s)} \left( \sum_{w \in \cW_\om^{k}\with wb \in \cW_\om^{k+1}} e^{\phi_k^\om([w])} \right) e^{\phi^{\te^k \om}([b])} \mu_{\te^{k+1}\om, s }(T_{\te^k\om}([b]_{\te^k\om}))\\
 & \leq \frac{s^{k+1} P_{\te^{k+1}\om}(s)}{P_\om(s)} \left(\prod_{l=0}^{k-1} M_{\te^l\om}\|\right)  e^{\phi^{\te^k \om}([b])} \leq \left( {\textstyle \prod_{l=0}^{k-1} \frac{M_{\te^l\om}}{m_{\te^l\om}}  }\right) e^{\phi^{\te^k \om}([b])},
 \end{align*}
 where the last inequality follows from (S2) and (\ref{eq:estimate_quotient}). By the same arguments, it follows that  
 \[ \Sigma_2^\om(b) \leq \left( {\textstyle \prod_{l=0}^{k-1}  M_{\te^l\om}} \right) e^{\phi^{\te^k \om}([b])}. \]
 Finally, for $n = 1, \ldots, k$, note that the set $E_n^\om \cap T_\om^{-k}([b]_{\te^k\om})$ is nonempty for at most one $b \in \cW^1_{\te^k\om}$. Hence, there exists $c_{\om,k} \leq \infty$ with $ \sum_{b > c_{\om,k}} \Sigma_1^\om(b) =0$.

For  a given $\epsilon > 0$, choose  a triple $(C, s_0, \Om')$ with $C>0$, $s_0 \in (0,1)$ and $\Om' \subset \Om$ such that $P(\Om')>1-\epsilon$ and  $P_\om(s)\geq C$ for all $s\geq s_0$, $\om \in \Om'$. For $\om \in \Om'$ and $c \geq c_{\om,k}$, we hence have that 
\begin{align} \nonumber
\sum_{b\geq c} \mu_{\om,s}(A_\om^{k,b}) & =  \frac{1}{P_\om(s)} \sum_{b\geq c}  \left(\Sigma_2^\om(b) + \Sigma_3^\om(b)\right)\\
 & \leq \left( C^{-1}{\textstyle \prod_{l=0}^{k-1}  M_{\te^l\om}} + {\textstyle \prod_{l=0}^{k-1} \frac{M_{\te^l\om}}{m_{\te^l\om}}} \right) \sum_{b\geq c}  e^{\phi^{\te^k \om}([b])} \label{eq:independence_of_s} \\
 & \leq B_{\te^{k+1} \om} \left( C^{-1}{\textstyle \prod_{l=0}^{k-1}  M_{\te^l\om}}+ {\textstyle \prod_{l=0}^{k-1} \frac{M_{\te^l\om}}{m_{\te^l\om}}} \right) M_{\te^k\om} < \infty.  \label{eq:summable}
\end{align}
Combining the summability in (\ref{eq:summable}) with the independence of the estimate from $s$ in (\ref{eq:independence_of_s}) then gives rise to the existence of $c^\ast_{\om,k} \geq c_{\om,k}$, for $\om \in \Om'$ and $k\in \bbN$,  such that $c^\ast_{\om,k} < \infty$ and 
\[ \sum_{b\geq c^\ast_{\om,k}}  \mu_{\om,s}(A_\om^{k,b}) \leq  \frac{\epsilon}{2^k}.\]
Observe that $\om \to c^\ast_{\om,k}$ might chosen to be measurable, which can be seen e.g. by construction of $c^\ast_{\om,k}$ as a maximum. For
$ K:= \{(\om,(x_0,x_1,\ldots))\with \om \in \Om', x_k < c^\ast_{\om,k}\}$
it then follows that
\begin{align*}
\int \mu_{\om,s}(K^c)dP &\leq \epsilon  + \int_{\Om'} \mu_{\om,s}(K^c)dP \\
& = \epsilon  + \int_{\Om'} \mu_{\om,s}(\{(x_0, \ldots) \with\exists k \hbox{ s.t. }x_k \geq c^\ast_{\om,k}\})dP  \\
&\leq \epsilon  + \int_{\Om'} 
 \sum_{\genfrac{}{}{0pt}{}{k=1, \ldots, \infty}{b \geq c^\ast_{\om,k} }} \mu_{\om,s}(A_\om^{k,b})dP   \leq \epsilon  +  P(\Om') \sum_{k=1}^\infty \frac{\epsilon}{2^k} \leq 2 \epsilon.
\end{align*}
Since $K\cap X_\om$ is compact for a.e. $\om \in \Om'$, the assertion follows. \end{proof}

As an immediate consequence of Crauel's relative Prohorov theorem we obtain that there exist a sequence $(s_n)$ with $s_n \nearrow 1$ and a random probability measure $\{\mu_\om\}$ such that   
\[ \lim_{n\to\infty}\int f d\mu_{\om,s_n} dP(\om)= \int f d\mu_\om dP(\om)\]
for all $f \in {\mathcal L}^C_1(P)$ where
$ {\mathcal L}^C_1(P):=\{f:X\to \bbR:\; f\arrowvert_{X_\om}\in C(X_\om),
\int \|f_{|X_\om}\|_\infty dP(\om)<\infty\}$ and $C(X_\om)$ denotes the set of continuous functions defined on $X_\om$. The following theorem is now stated without the assumption that $P_G(\phi)=0$.

\begin{theorem}\label{theo:conformal} Let $(X,T,\phi)$ be a topologically mixing system of divergence type with (H2), S(1-2) and $P_G(\phi)>-\infty$.
 Then there exists  a random probability measure $\{\mu_\om\}$ such that 
 $\mu_\om$ is positive on cylinders for a.e. $\om \in \Om$ and, for $x \in X_\om$, 
\[ \frac{d\mu_{\te \om}\circ T_\om}{d\mu_\om}(x) = \la(\om)^{-1} e^{P_G(\phi)-\phi^\om(x)}.\]
\end{theorem}

\begin{proof} By substituting $\phi$ with $\phi - P_G(\phi)$, we may assume without loss of generality that $P_G(\phi)=0$.
Let $(t_k : k \in \bbN)$ be the sequence given by Lemma 4.1. By Lemma 4.2, there exists a subsequence $(s_n:n \in \bbN)$ of $(t_k)$ and a random probability measure $\{\mu_\om: \om \in \Om \}$ which is the limit of $\{\mu_{\om,s_n} \}$. For $k \in \bbN$,  $a \in \cW^k_\om$ and $A \subset [a]_\om$ such that $A=[ab]_\om$ for some $b \in \cW_{\te^k\om}^\ast$, it follows as in the proof of proposition 6.3 in \cite{DenkerKiferStadlbauer:2008} that   
  \begin{align*} \mu_{\om,s}(A)  & =
   \frac{1}{P_\om(s)}\sum_{ \genfrac{}{}{0pt}{}{k \in J_\om(\widetilde{\Om}), k\leq n }{ x \in E^{k}_{\om}\cap A}}   
 s^k e^{\phi^k_\om(x)} \\
 & + \frac{P_{\te^k\om}(s)}{P_\om(s)}\frac{1}{P_{\te^k\om}(s)}
   \sum_{\genfrac{}{}{0pt}{}{n \in J_\om(\widetilde{\Om}), k>n, }{ x \in E^{k-n}_{\te^n\om}\cap T^n_\om(A)}}  s^k e^{\phi_n^\om(\tau_a(x))}  e^{\phi_{k-n}^{\te^n\om}(x)}.
   \end{align*}
Observe that the first summand tends to zero when $s \to \rho$. Hence, for  $B \subset \Om_a$, we obtain by integration that
\begin{eqnarray*} 
 \displaystyle \int_B \mu_{\om}(A) dP(\om)  &=&  \int_B \lim_{n \to \infty} \mu_{\om,s_n}(A) dP(\om) \\
&=& \displaystyle   \lim_{n \to \infty} 
\int_B \frac{P_{\te^k\om}(s_n)}{P_\om(s_n)} \int_{T_\om^k(A)} e^{\phi_n^\om(\tau_a(x))} d\mu_{\te^k\om,s_n}(x) dP(\om)
 \end{eqnarray*}
By considering  a subset of $B$, we may assume that $M_\om M_{\te\om} \cdots M_{\te^{k-1}\om}$ is uniformly bounded from above on $B$. Then,
\begin{eqnarray*} 
(\star)& := & \int_B  \La_k(\om) \int_{T_\om^k(A)} e^{\phi_n^\om \circ \tau_a} d\mu_{\te^k\om} dP  - \int_B \mu_{\om}(A) dP\\
&=&  
\lim_{n \to \infty} \int_B \left( \frac{P_{\te^k\om}(s_n)}{P_\om(s_n)} - \La_k(\om) \right) \int_{T_\om^k(A)} e^{\phi_n^\om \circ \tau_a} d\mu_{\te^k\om} dP \\
& &+  
\lim_{n \to \infty} \int_B  \frac{P_{\te^k\om}(s_n)}{P_\om(s_n)} \int_{T_\om^k(A)} e^{\phi_n^\om \circ \tau_a} d(\mu_{\te^k\om,s_n}  - \mu_{\te^k\om})  dP. \\
&\geq &  \lim_{n \to \infty} \int_B \left( \log \frac{P_{\te^k\om}(s_n)}{P_\om(s_n)  \La_k(\om)} \right) \La_k(\om) \int_{T_\om^k(A)} e^{\phi_n^\om \circ \tau_a} d\mu_{\te^k\om} dP  =0.
 \end{eqnarray*} 
By the same argument, it follows that $(\star) \leq 0$. Now observe that, since $X$ is a subset of the product space $\bbN^\bbN \times \Om$, the measure $m$ defined by $m(A) := \int \mu_{\te\om}(T_\om(A)) dP$ for sets $A$ such that $T|_A$ is one-to-one can be disintegrated with respect to $P$. Hence, by uniqueness of the Radon-Nikodym derivative, it follows that 
\begin{equation} \nonumber 
\frac{d\mu_{\te^k \om}\circ T^k_\om}{d \mu_\om}(x) = e^{-\phi^\om_k(x)} \La_{k}(\om)^{-1}.  \end{equation} 
In order to show that $\{\mu_\om\}$ is positive on cylinders, choose $a \in \cW$ and $\Om' \subset \Om_a$ of positive measure such that $\mu_\om([a]_\om)>0$ for all $\om \in \Om'$. It follows, for a.e. $\om \in \Om$, $n \in \bbN$ and $w \in \cW^n_\om$, that there exists $k > n$ such that $\te^{k}\om \in \Om'$ and $T_\om^{k}([w]) \supset [a]_{\te^k\om}$. In particular, since ${d\mu_{\te^k \om}\circ T^k_\om}/{\mu_\om}(x) >0$ for all $x \in [w]_\om$, we have $\mu_\om([w])>0$.
\end{proof}

\begin{remark} \label{rem:missing_statement}
The theorem no longer states that $P_\om(s_n)/P_{\te \om}(s_n) \xrightarrow{n \to \infty} \la(\om)$ (theorem 4.1 in \cite{Stadlbauer:2010}). Since the statement also follows from a stronger result in the forthcoming paper \cite{Stadlbauer:2013b}, its proof will be omitted. Since the convergence was used in the proof of corollary \ref{cor:pr}, we give an independent and much simpler proof of that result.
\end{remark}

\begin{remark} \label{rem:conformal}
 For $a \in \cW_\om$ and $A \subset [a]_\om$, the above result implies that 
\[ \mu_{\te\om}(T_\om(A)) =  \la(\om) \int_A e^{P_{G}(\phi) - \phi^\om} d\mu_\om \]
for a.e. $\om \in \Om$. Hence $\{\mu_\om\}$ is a {$(\la \exp(P_{G}(\phi) - \phi))$-random conformal measure}. Moreover, note that conformality gives a characterization of $\{\mu_\om\}$ as eigenmeasure of the dual $(L^\phi_\om)^\ast$ which acts on the space of Radon measures (see e.g. \cite{DenkerKiferStadlbauer:2008}), that is,
\[ (L^\phi_\om)^\ast(\mu_{\te\om}) = \la(\om)e^{P_G(\phi)}\mu_\om.\]
\end{remark}
\begin{remark} Now assume that $\phi$ has property (H1). For $a=(a_0, \ldots, a_{n-1}) \in \cW^n_\om$ we  then immediately obtain an estimate 
 for $\mu_\om([a]_\om)$ in terms of the measure of $T_{\te^{n-1}\om}([a_{n-1}]_\om)$. Set $\La_n(\om) := \la(\om)\cdot \la(\te \om)\cdots \la(\te^{n - 1}\om)$. We then have
\[\frac{1}{B_{\te^n\om}}   \mu_{\te^n\om}( T_{\te^{n-1}\om}([a_{n-1}]_\om)) \leq \La_n(\om) \frac{\mu_\om([a]_\om)}{e^{\phi^n_\om(x)-nP_G(\phi)}}  \leq B_{\te^n\om} \mu_{\te^n\om}( T_{\te^{n-1}\om}([a_{n-1}]_\om)),\]
 for all $x \in [a]_\om$. If $(X,T)$ has the big image property, then  
\[ D_{\om}:= \inf\{ \mu_{ \om}( T_{\te^{-1} \om}([b]_{\te^{-1} \om}))  \with b \in \cW^1_{\te^{-1} \om}\} >0\]
for all  $\om  \in \Om_{\hbox{\tiny bi}}$. Hence, for a.e. $\om \in \Om$ and $n$ with $\te^{n}\om \in \Om_{\hbox{\tiny bi}}$, 
\[ (B_{\te^n\om})^{-1}  D_{\te^n\om}  \leq \La_n(\om) \frac{\mu_\om([a]_\om)}{e^{\phi^n_\om(x)-nP_G(\phi)}}  \leq B_{\te^n\om},\]
which is a natural analogue of the Gibbs property for random topological Markov chains.
\end{remark}

We proceed with applications of the above theorem to systems with the b.i.p.-property. In this case, by Theorem \ref{theo:ergodic}, the system is of divergence type and the above theorem is applicable. 

\begin{corollary} \label{cor:pr} If $(X,T,\phi)$ has the b.i.p.-property and (H$^\ast$) and (S1-2) hold, then there exist positive measurable functions $K, K^\ast:\Om_a \to \bbR$, $\cN:\Om_a \to \bbN$ such that, for all $\om \in \Om_a$ and $n\geq \cN(\om)$ with  $\om\in \Om_a\cap\te^{n}\Om_a$,
\[ K(\om) K({\te^{-n}\om}) \leq\frac {Z_n^{\te^{-n}\om}(a)}{\La_n(\te^{-n}\om) e^{nP_G(\phi)}}\leq K^\ast(\om).   \]
\end{corollary}

\begin{proof}Since 
$\sum_{w \in \cW^n_\om} \mu_\om([w])=1$, it follows from 
 remark \ref{rem:conformal} and the definition of $A_n^\om$ that, for a.e. $\om$ and $n \in \bbN$ with $\te^n \om \in \Om_{\hbox{\tiny bi}}$, 
\[ 
B_{\te^n \om }^{-1} \leq
\frac{A_n^\om}{ \La_n(\om)  e^{n P_G(\phi)}} =
\frac{\sum_{w \in \cW_\om^n} e^{\phi_n^\om([w])} }{ \La_n(\om)  e^{n P_G(\phi)}} 
\leq  B_{\te^n \om } D_{\te^n(\om)}^{-1} 
.\]
The corollary is then a consequence of lemma \ref{lem:partition_functions}.
\end{proof}

By choosing a subset $\Om_r$ of $\Om_a$ for which $K(\om)$ is uniformly bounded, it immediately follows that there exists $\widetilde{K}:\Om_r \to \bbR$, $\widetilde{K}>0$ with  
\[ \widetilde{K}^{-1}(\om) \leq\frac {Z_n^{\te^{-n}\om}(a)}{\La(\te^{-n}\om)}\leq \widetilde{K}(\om),\]
for a.e. $\om \in \Om_r$ and $n \geq \cN(\om)$ with $\te^{-n}\om \in \Om_r$. In particular, $(X,T,\phi)$ is positive recurrent as introduced in \cite{DenkerKiferStadlbauer:2008} which is a sufficient condition for the relative Ruelle-Perron-Frobenius theorem (cf. Theorem 5.3 in there). For the definition of relative exactness in the following statement, we refer to \cite{Guivarch:1989,DenkerKiferStadlbauer:2008}.

\begin{theorem}[cf. \cite{DenkerKiferStadlbauer:2008}] \label{theo:rpf}  Assume that $(X,T,\phi)$ has the b.i.p.-property and (H $^\ast$) and S(1-2) hold. Then there exists a measurable family of functions $(h^\om : \om \in \Om)$ such that, for $\mu$ and $\la$ given by Theorem \ref{theo:conformal}, the following holds.
\begin{enumerate}
\item For a.e. $\om \in \Om$, $h^\om:X_\om \to \bbR$ is a positive, $1$-H\"older continuous function which is bounded from above and below on cylinders.
\item For a.e. $\om \in \Om$, we have $ L^\om_\phi h^\om=\la(\om)e^{P_G(\phi)}h^{\te\om}$, $\int h^\om \mu_\om=1$.
\item The random topological Markov chain is relatively exact with respect to $(\mu_\om)$. In particular,
for $\{f^\om:\om \in \Om'\}$ with $f^\om  \in L^1(\mu_\om)$ for a.e. $\om \in \Om$, we have
\[  \lim_{n \to \infty} \left\| \frac{L^{\om,n}_{\phi}f^\om}{\La_n(\om)e^{nP_G(\phi)}} - h^{\te^n\om}\int f^\om d\mu_\om\right\|_{L^1(\mu_{\te^n \om})} =0. \]  
\item The probability measure given by $h^\om d\mu_\om dP$ is $T$-invariant and ergodic.
\end{enumerate}
\end{theorem}

\begin{proof} These are immediate consequences of  Theorem 5.3,  Proposition 7.3 and Proposition 7.4 in \cite{DenkerKiferStadlbauer:2008}.
\end{proof}

\begin{remark} \label{rem:finite} Recall e.g. from \cite{BogenschutzGundlach:1995}, that a random subshift of finite type is a random topological Markov chain with $\ell_\om < \infty$. Now assume that a random subshift $(X,T)$ of finite type is topologically mixing and has properties (H2) and (S1-2). Clearly, $(X,T)$ has the b.i.p.-property. Moreover, it easily can be seen that $V_1^\om(\phi) < \infty$.
Hence the above Theorem is an extension of Ruelle's theorem in \cite{BogenschutzGundlach:1995}. 
\end{remark}

Furthermore, by considering a potential which is constant on cylinders of length two, we obtain a Perron-Frobenius-theorem for the following class of random matrices. Let
$A = \{A_\om \with \om \in \Om\}$ with  $A_\om=\big(p_{ij}^\om,\, i < \ell_\om,j < \ell_{\te\om}\big) $ and  $p_{ij}\geq 0$ a.s. be a measurable family of random matrices. We then refer to $A$ as 
a summable random matrix with the b.i.p.-property if
\begin{enumerate}
  \item the signum of $A$ gives rise to a random topological Markov chain with the b.i.p.-property, 
  \item for a.e. $\om \in \te^{-1}(\Om_{\textrm{\tiny bi}} \cup \Om_{\textrm{\tiny bp}})$, 
  we have 
  \[\sup \left\{ \frac{p_{ij}^\om}{p_{ik}^\om} \with i<\ell_\om, j,k < \ell_{\te\om}, p_{ik}^\om \neq 0 \right\}< \infty,\] 
  \item there exist positive random variables $\om \mapsto m_\om$ and $\om \mapsto M_\om$ with $\log m, \log M \in L^1(P)$ such that, for a.e. $\om \in \Om$,
  \[m_\om \leq \inf_{j < \ell_{\te\om}} \sum_{i < \ell_\om} p_{ij}^\om \leq \sup_{j < \ell_{\te\om}} \sum_{i < \ell_\om} p_{ij}^\om \leq M_\om. \]
\end{enumerate}
By viewing $A$ as a locally constant potential we arrive at the following random Perron-Frobenius theorem. Below, $\bbR^{\infty-1}$ stands for $\bbR^\bbN$, and $(B)_{ij}$ for the coefficient of the matrix $B$ with coordinates $(i,j)$.

\begin{corollary} \label{cor:pf}
For a summable random matrix $A$ with the b.i.p.-property, there exist a positive random variable $\la:\Om \to \bbR$ and strictly positive random vectors $h=\{h^\om \in \bbR^{\ell_\om - 1} \with \om \in \Om\}$ and $\mu=\{\mu^\om \in \bbR^{\ell_{\te\om} - 1} \with \om \in \Om\}$ such that, for a.e. $\om \in \Om$,
\[ (h^\om)^t A^\om = \la(\om) h^{\te\om}, \quad  A^\om \mu^{\te \om} =  \la(\om)\mu^{\om}, \quad (h^\om)^t\mu^\om =1.\]
Furthermore, for a.e. $\om \in \Om$ and $i < \ell_\om$, we have 
\[\lim_{n\to \infty} \sum_{j < \ell_{\te^{n}\om}} \left| \frac{(A^\om\cdot A^{\te\om} \cdots  A^{\te^{n-1}\om})_{ij}}{\La_n(\om)} - \mu^\om_i h^{\te^{n}\om}_j\right| \mu^{\te^{n}\om}_j =0. \] 
\end{corollary}

\begin{proof} Let $(X,T)$ be the random topological Markov chain given by the signum of $A$ and, for $x \in [a_0a_1]_\om$, set $\phi^\om(x):= \log p_{a_0a_1}^\om$. Then $\phi$ is 2-Hölder continuous and, by condition (ii), is 1-Hölder continuous for $\om \in \te^{-1}(\Om_{\textrm{\tiny bi}} \cup \Om_{\textrm{\tiny bp}})$. As a consequence of the summability assumption (iii), it then follows that Theorem \ref{theo:rpf} is applicable to $(X,T,\phi)$. So let $\la'$, $h'$ and $\mu'$ be given by this result. The random variable $\la$ is then defined  by $\la := \la' e^{P_G(\phi)}$. Furthermore, since $L_\phi$ acts on functions which are constant on cylinders, it follows by the construction of the eigenfunction in Proposition 7.3 in \cite{DenkerKiferStadlbauer:2008} that $h'$ is constant on cylinders of length 1. Hence, with $h$ given by $h^{\om}_a := h'|_{[a]_\om}$, we have that, for a.e. $\om \in \Om$ and $x \in [b]_{\te\om}$, 
\[ ((h^\om)^t A^\om)_b = L_\phi^\om(h')(x) = \la(\om) h'(x) =   \la(\om) h^{\te \om}_b.\]
Furthermore, for $\mu$ given  by $\mu^{\om}_a := \mu'_\om([a]_\om)$, the identity $ A^\om \mu^{\te \om} =  \la(\om)\mu^{\om}$ follows by similar arguments. The remaining assertion is an application of Theorem \ref{theo:rpf} (iii) to the indicator function $1_{[a]_\om}$.
\end{proof}

As a concluding remark, we give an application of our results to the existence of a stationary vector (or stationary distribution) for a stationary Markov chain with countably many states in a stationary environment. 
Recall that such a Markov chain is given by a \emph{random stochastic matrix} $A= \{(p^\om_{ij} :i < \ell_\om,j < \ell_{\te\om})\with \om \in \Om\}$, that is, $\sum_{j < \ell_{\te\om}} p^\om_{ij}=1$ for every $i < \ell_{\te\om}$ and  a.e. $\om \in \Om$, where $p^\om_{ij}\geq  0$ stands for the random transition probability from state $i$
to $j$. Moreover, a random vector $\pi=\{ (\pi_i^\om \with i < \ell_\om) \with \om \in \Om \}$ is called \emph{random stochastic vector} if $\pi \geq 0$ and  $\sum_i \pi_i^\om =1$ for a.e. $\om \in \Om$. If, in addition, $\pi^\om A^\om = \pi^{{\te\om}}$ for a.e. $\om \in \Om$, then $\pi$ is referred to as a 
\emph{random stationary distribution}.
\begin{corollary}\label{theo:stochastic:matrix} 
Assume that $A$ is a  random stochastic matrix such that
\begin{enumerate}
  \item the signum of the transpose $A^t$ of $A$ defines a random topologically mixing topological Markov chain with the big preimages property,
  \item for a.e. $\om \in \te(\Om_{\textrm{\tiny bi}})$, we have 
   \[\sup \left\{ \frac{p_{ji}^\om}{p_{ki}^\om} \with j,k < \ell_{\om},, i<\ell_{\te\om},  p_{ki}^\om \neq 0 \right\}< \infty.\]
\end{enumerate}
Then there exists a unique random stationary distribution $\pi$. In particular, $\pi >0$ 
and, for an arbitrary random vector $\{(f^\om_i: i < \ell_\om)\with \om \in \Om\}$ with $\sum_i |f^\om_i| \pi^\om_i < \infty$, we have
\[\lim_{n\to \infty} \sum_{j < \ell_{\te^{-n}\om}} \pi^{\te^{-n}\om}_j  \left| {(A^{\te^{-n}\om} \cdots A^{\te^{-2}\om} \cdot  A^{\te^{-1}\om} f^\om)_{j}} - \sum_i f^\om_i \pi^\om_i  \right| =0.  \]
\end{corollary} 

\begin{proof} By assumption, the signum of $A^t$ defines a system $((X,T),(\Om,\te^{-1}))$ which s  topological mixing and has the big preimages property. Moreover, as a consequence of (ii), the potential defined by $\phi^\om\arrowvert_{[ij]} := \log p^\om_{ji}$ is 1-Hölder continuous for $\om \in \te(\Om_{\textrm{\tiny bp}})$. 

Since $A$ is a stochastic matrix, it follows that the constant function $1$ is an eigenfunction of $L_\phi$, and hence $\cZ^\om_n = 1$ for all $n \in \bbN$ and a.e. $\om \in \Om$.  For $a \in \cW^1$, it follows from Lemma \ref{lem:partition_functions} (i) that there exists $C_\om >0$ such that $\cZ^\om_n(a) \cdot C_\om \geq 1$ for all $n \in J_\om(\Om_a)$ sufficiently large. 
Hence $P_G(\phi)=0$, $(X,T,\phi)$ is of divergence type and positive recurrent.
The assertion then follows using similar arguments as in the proof of Corollary \ref{cor:pf}.
\end{proof}

\begin{remark}  It follows from Remark \ref{rem:conformal}, with $\{\mu_\om\}$ referring to the invariant measure associated with $((X,T),(\Om,\te^{-1}))$ 
and for $v=(v_0 \ldots v_n) \in \cW^{n+1}_\om$,  that  
\[ \mu_\om([v]) = \pi^{\te^{-n}\om}_{v_n} p^{\te^{-n}\om}_{v_nv_{n-1}} \cdots  p^{\te^{-1}\om}_{v_1v_{0}} .\]
 Hence, $((X,T),(\Om,\te^{-1}))$ might be seen as the time reversal of a (probabilistic)  Markov chain in random environment. However, for a random stochastic vector $\{ \nu^\om \}$, it follows that $\{\nu^\om A^\om\}$ can be recovered from, for $i \in \cW_{\te\om}$,
\[ (\nu^\om A^\om)_i = \int_{[i]} \nu^\om \circ T^{\te \om} d\mu_{\te \om} = \int L_\phi^{\te\om}(\1_{[i]}) \nu^\om  d\mu_\om. \]
Hence, the above result is closely related to Problem 5.7 in \cite{Orey:1991}.
\end{remark}

\section*{Acknowledgements}
The author acknowledges support by \emph{Fundação para Ciência e a Tecnologia} through grant SFRH/BPD/39195/2007 and the \emph{Centro de Matemática da Universidade do Porto}.


\begin{thebibliography}{10}

\bibitem{BogenschutzGundlach:1995}
T. Bogensch{\"u}tz and V.~M. Gundlach, Ruelle's transfer operator for random
  subshifts of finite type, \emph{Ergod. Th. Dynam. Sys.} \textbf{15} (1995)
  413--447.
  
\bibitem{Crauel:2002}
H. Crauel, \emph{Random probability measures on {P}olish spaces},
  \emph{Stochastics Monographs}, Vol.~11 (Taylor \& Francis, London, 2002).

\bibitem{DenkerKiferStadlbauer:2008}
M. Denker, Y. Kifer and M. Stadlbauer, Thermodynamic formalism for random
  countable Markov shifts, \emph{Discrete Contin. Dyn. Syst.} \textbf{22}
  (2008) 131--164.

\bibitem{DenkerUrbanski:1991b}
M. Denker and M. Urba{\'n}ski, On the existence of conformal measures,
  \emph{Trans. Am. Math. Soc.} \textbf{328} (1991) 563--587.

\bibitem{Derriennic:1983}
Y. Derriennic, Un th{\'e}or{\`e}me ergodique presque sous-additif, \emph{Ann.
  Probab.} \textbf{11} (1983) 669--677.
  
\bibitem{Guivarch:1989}
Y. Guivarc'h, Propri\'et\'es ergodiques, en mesure infinie, de certains
  syst\`emes dynamiques fibr\'es, \emph{Ergod. Th. Dynam. Sys.} \textbf{9}
  (1989) 433--453.


\bibitem{Kifer:1996a}
Yu. Kifer, Perron-Frobenius theorem, large deviations, and random perturbations
  in random environments, \emph{Math. Z.} \textbf{222} (1996) 677--698.

\bibitem{Kifer:2008}
 Yu. Kifer, Thermodynamic formalism for random transformations revisited,
  \emph{Stoch. Dyn.} \textbf{8} (2008) 77--102.
  
\bibitem{MauldinUrbanski:2001}
R.~D. Mauldin M. and Urba{\'n}ski, Gibbs states on the symbolic space over an
  infinite alphabet, \emph{Israel J. Math.} \textbf{125} (2001) 93--130.

\bibitem{Orey:1991}
S. Orey, Markov chains with stochastically stationary transition
  probabilities, \emph{Ann. Probab.} \textbf{19} (1991) 907--928.

\bibitem{Sarig:1999}
 O. Sarig, Thermodynamic formalism for countable Markov shifts, \emph{Ergod.
  Th. Dynam. Sys.} \textbf{19} (1999) 1565--1593.


\bibitem{Sarig:2003}
O. Sarig, Existence of Gibbs measures for countable Markov shifts, \emph{Proc.
  Amer. Math. Soc.} \textbf{131} (2003) 1751--1758.

\bibitem{Stadlbauer:2010}
M.~Stadlbauer.
\newblock On random topological {M}arkov chains with big images and preimages,
\newblock {\em Stoch. Dyn.}  \textbf{10} (2010) 77--95.

\bibitem{Stadlbauer:2013b}
M.~Stadlbauer. \newblock Coupling methods for random topological markov chains,  \href{http://arxiv.org/abs/1312.6033}{\em Arxiv 1312.6033}
(2013).


\bibitem{StratmannUrbanski:2007}
 B.~O. Stratmann and M. Urba{\'n}ski, Pseudo-{M}arkov systems and infinitely
  generated {S}chottky groups, \emph{Amer. J. Math.} \textbf{129} (2007)
  1019--1062.

\bibitem{Thaler:1983}
M. Thaler, Transformations on $[0,1]$ with infinite invariant measure,
  \emph{Israel J. Math.} \textbf{46} (1983) 67--96.

\end{thebibliography}
\end{document}